\documentclass{siamart0216}
\usepackage{amssymb}
\usepackage{amsmath}
\usepackage{subfigure}
\usepackage{color}
\usepackage{algorithm}
\usepackage{algorithmic}

\newtheorem{remark}[theorem]{Remark}
\newcommand{\R}{\ensuremath{\mathbb{R}}}
\renewcommand{\P}{\ensuremath{\mathbb{P}}}
\newcommand{\E}{\ensuremath{\mathbb{E}}}
\newcommand{\Var}{\ensuremath{\mathbb{V}}}
\newcommand{\M}{\ensuremath{\mathbb{M}}}
\newcommand{\var}{\Var}

\newcommand{\Ord}{{\mathcal{O}} }
\newcommand{\F}{{\mathcal{F}} }
\newcommand{\N}{\mathbb{N}}

\newcommand{\w}{{\mathbf{w}} }
\newcommand{\U}{{\mathbf{U}} }

\newcommand{\V}{{\mathbf{V}} }
\newcommand{\Z}{Z}

\newcommand{\X}{{\mathbf{x}} }
\newcommand{\Y}{{\mathbf{y}} }
\newcommand{\Q}{{\mathbf{Q}}}
\newcommand{\x}{{\mathbf{x}} }
\newcommand{\A}{{\mathbf{A}} }

\renewcommand{\S}{{\mathbf{S}} }
\renewcommand{\phi}{\varphi}
\newcommand{\phibf}{\pmb{\varphi}}

\newcommand{\Div}{\mathrm{div}}
\renewcommand{\div}{\mathrm{div}}

\def\e{{\mathrm{e}}}

\newcommand{\Comment}[1]{}

\newcommand{\eps} {\varepsilon}
\renewcommand{\i}{\ifmmode\mathit{\mathchar"7010 }\else\char"10 \fi}
\renewcommand{\j}{\ifmmode\mathit{\mathchar"7011 }\else\char"11 \fi}

\newcommand{\Dx}{\Delta x}
\newcommand{\Dy}{\Delta y}
\newcommand{\Dz}{\Delta z}
\newcommand{\Dt}{\Delta t^n}

\newcommand{\cA}{\mathcal{A}}

\newcommand{\cN}{\mathcal{N}}

\newcommand{\rL}{\mathbf{L}}
\newcommand{\W}{\mathbf{W}}

\begin{document}

\title{Uncertainty quantification for hyperbolic\\conservation laws with flux coefficients\\given by spatiotemporal random fields}

\author{
        Andrea Barth$^{1}$
        and
        Franz G. Fuchs$^{2}$
        \\[.5em]
        \tiny
         $^1$SimTech, University of Stuttgart, Pfaffenwaldring 5a, 70569 Stuttgart, Germany
         $^2$SINTEF, Forskningsveien 1, 0314 Oslo, Norway
       }

\date{\today}

\maketitle

\begin{abstract}
    In this paper hyperbolic partial differential equations with random coefficients are discussed.
    We consider the challenging problem of flux functions with coefficients modeled by spatiotemporal random fields.
    Those fields are given by correlated Gaussian random fields in space and Ornstein--Uhlenbeck processes in time.
    The resulting system of equations consists of a stochastic differential equation for each random parameter coupled to the hyperbolic conservation law.
    We define an appropriate solution concept in this setting and analyze errors and convergence of discretization methods.
    A novel discretization framework, based on Monte Carlo Finite Volume methods, is presented for the robust computation of moments of solutions to those random hyperbolic partial differential equations.
    We showcase the approach on two examples which appear in applications: The magnetic induction equation and linear acoustics, both with a spatiotemporal random background velocity field.
\end{abstract}

\footnotetext[3]{Accepted for publication in SIAM Journal on Scientific Computing (SISC) "Methods and Algorithms for Scientific Computing".}


\setlength{\tabcolsep}{.0em}
\section{Introduction}

Hyperbolic partial differential equations with random data have been an active research field over the last decades.
In ample situations measurements are not accurate enough to allow an exact description of a physical phenomena by a deterministic model.  
To account for this, uncertainty is introduced in the appropriate parameters and the distribution of the (now stochastic) solution is studied.
In this paper we consider linear hyperbolic PDEs with time and space dependent flux functions.
Important examples of such equations include linear elasticity, and linear shallow water equations, as well as the linear acoustics and the magnetic induction equation considered in this article.

\subsection{Linear hyperbolic conservation/balance laws}
Many important physical phenomena can be modeled by first order linear hyperbolic systems.
We consider a system of the general form
\begin{equation}\label{eq:general_hyp}
    \begin{cases}
        \partial_t \U(\x,t) + \vec\div \cdot \left(\vec\A(\x,t) \ \U(\x,t)\right) = \S(\x,t),\\
        \hfill \U(\x,0) = \, \U_0(\x),
    \end{cases} \quad \x \in D\subset\R^d, t>0,
\end{equation}
with suitable boundary conditions.
Here $\U(\X,t)$ denotes the vector of conserved quantities at a point $\X$ in the domain $D$ at time $t$,
and $\vec\A \U  = (\A^{x_1} \U, \cdots, \A^{x_d} \U)$, for $d\in\N$, are the flux functions in the $x_1,\cdots, x_d$--direction, respectively,
where $\A^{x_r} : \R^d \rightarrow \R^d$ are linear maps.
The matrix $\vec \A$ is assumed to be diagonalizable with real eigenvalues.
Well-posedness of linear non-autonomous systems of conservation laws is studied for instance in \cite{godlewski1991hyperbolic, godlewski1995numerical, gustafsson1995time, LEV1, wlokapartial}.
For a linear advection equation, even in the case of variable coefficients, the characteristics never cross \cite{LEV1}.
Looking at the linear advection equation $u_t+(x u)_x=0$ in one (spatial) dimension, one can prove that all characteristic converge to the origin.
It becomes clear that solutions for $t\rightarrow\infty$ may become unbounded and contain Dirac delta functions.

In this article we treat the case where the flux functions $\vec \A(\x,t,\V)$ depend on $m\in\N$ coefficients $\V=(v_1,\cdots,v_m)$, that are modeled as correlated random fields in space, and stochastic processes in time, defined on a probability space $(\Omega,\cA,\P)$.
We assume from now on that the random fields are \emph{($\P$-almost surely) differentiable with respect to $\x$}, \cite{adler2009random}.
The stochastic processes considered provide the coefficients implicitly, namely \emph{as the solution} to a stochastic differential equation (SDE).
This means that the system we are considering is given by
\begin{equation}\label{eq:general_hyp_stoch}
    \begin{cases}
        \partial_t \U(\x,t) + \vec\div \cdot \left(\vec\A(\x,t,\V(\X,t)) \ \U(\x,t)\right) &= \S(\x,t),\\
        \partial_t \V(\X,t) &= \boldsymbol{\mu}(\V(\X,t)) + \boldsymbol{\sigma}(\V(\X,t)) \ \boldsymbol{\nu}(\X,t), \\
    \end{cases}
\end{equation}
where the vectors $\boldsymbol{\mu}$, and $\boldsymbol{\sigma}$ are arbitrary functions, and the vector $\boldsymbol{\nu}$ is a \emph{random function in space and time}, often referred to as the \emph{"noise term"}.

In general, the It\^o-SDE for the parameter $\V$ in Equation~\eqref{eq:general_hyp_stoch} has no explicit (strong or weak) solution.
Many stochastic schemes fall into the class of stochastic Runge-Kutta (SRK) methods.
Weak approximations focus on the expectation of functionals of solutions, whereas strong approximations are concerned with pathwise solutions.
For an overview on the theory of SDEs and numerical methods we refer to \cite{kloeden2012numerical, kloeden00, oksendal2013stochastic, milstein1995numerical, sobczyk2013stochastic} and references therein.

In general, there are no explicit solution formulas for deterministic variable-coefficient linear hyperbolic systems of the form \eqref{eq:general_hyp}, let alone for the stochastic PDE \eqref{eq:general_hyp_stoch}.
Numerical methods are therefore widely used to approximate the solutions of those types of equations.
Finite Difference, Finite Volume and Discontinuous Galerkin methods are popular approaches to obtain efficient time and space discretizations for the problem at hand, see \cite{LEV1,gustafsson1995time} and references therein.

\subsection{Uncertainty quantification}
In many applications the parameters of the flux functions in Equation~\eqref{eq:general_hyp} are determined by measurements. Then is, in fact, only statistical information available.
Among other phenomena, scarcity of measurements of material properties or background velocity fields lead to uncertainty in the parameters of the flux function.
Given the statistical description of the parameters, it is of interest to efficiently \emph{quantify} the resulting \emph{uncertainty} in the solution to Equation~\eqref{eq:general_hyp}.
An appropriate mathematical notion, together with a proof of existence and uniqueness, of random solutions for systems of hyperbolic conservation laws has recently been developed in e.g.~\cite{bijl2013uncertainty, vsukys2013multi}(see also references therein).

Efficient numerical methods for uncertainty quantification in the setting of partial differential equations has been intensively studied in recent years.
A non-exhaustive list of literature on uncertainty quantification for hyperbolic conservation laws includes~\cite{abgrall2008simple, chen2005uncertainty, lin2004stochastic, lin2006predicting, tryoen2010intrusive, poette2009uncertainty, wan2006long, gottlieb2008galerkin, JSK02, TZ10, poette2009uncertainty} and references therein.
Among the most popular techniques are stochastic Galerkin methods based on \emph{generalized polynomial chaos expansions} (gPC).
These methods reduce the stochastic model to a (high-dimensional) deterministic one. This comes to the price that they are highly intrusive, such that existing numerical schemes for conservation laws cannot be used.
A second class of methods for uncertainty quantification are stochastic collocation methods, which are non-intrusive and easier to parallelize than gPC based methods.
Solutions to hyperbolic conservation laws, however, do not have the necessary regularity with respect to the stochastic variables, which in general diminishes the use of both gPC and collocation methods.
There are a number of other techniques, namely stochastic Finite Volume methods (see~\cite{mishras}), adaptive analysis of variance, proper generalized decomposition, and Fokker-Planck-Kolmogorov type techniques.
The latter can handle low parametric regularity, but assume that the "effective" number of stochastic dimensions is low and require often impractical complex representations of the input random fields.

Here, we use \emph{Monte Carlo} (MC) methods to quantify the uncertainty, which comprises a class of non-intrusive methods well suited for problems with low parametric (stochastic) regularity.
Monte Carlo methods rely on repeated sampling of the probability/parameter space.
For each sample the underlying, (then) deterministic, PDE is solved and the sample solutions are combined to obtain statistical information in the form of moments of the distribution.
Monte Carlo methods are very robust with respect to the regularity of the solutions.
However, the convergence rate of a (plain) MC method is limited to $\frac{1}{2}$ with respect to the number of samples, which means that typically a large number of realizations of approximations of solutions to the underlying problem~\eqref{eq:general_hyp} has to be computed.
In the Monte Carlo approximation of moments of solutions to stochastic partial differential equations the discretization error consists of a \emph{statistical error} and a \emph{spatiotemporal} error of the numerical scheme, as can be seen in Equation~\eqref{eq:errorest}.
For non-autonomous linear systems of conservation laws it has been shown in~\cite{mishra2014multi}, that 
the number of Monte Carlo samples $M$ can be chosen
\begin{equation}\label{eq:mc_samples}
    M = \mathcal{O}(\Delta x^{-2o}),
\end{equation}
in order to equilibrate the statistical and spatiotemporal error of an underlying FV scheme with convergence rate $o>0$.
The computational complexity due to the slow convergence rate, can be lowered by \emph{Multilevel Monte Carlo} (MLMC) methods, which have been proposed in~\cite{vsukys2013multi} and related papers by the same authors~ \cite{mishra2012multilevel,mishra2012sparse,mishra2012multi}.
For a result on the convergence and computational complexity of the multilevel Monte Carlo approximation for general Hilbert-space-valued random variables see~\cite{BL12}.
The idea behind MLMC methods is to use a hierarchy of different levels of resolution of the underlying deterministic numerical solver.
The level dependent numbers of MC samples are then chosen, in such a way that the total computational work is minimal while the sum of all error terms is asymptotically optimized.
In the case of non-autonomous linear systems of conservation laws, see~\cite{mishra2014multi} for details.
Optimal computational complexity is not the main focus of this paper.
We remark that it is not particularly challenging to employ an MLMC method, as the problems considered here fall into the class of problems for which a general MLMC method is derived in~\cite{BL12}.

Random (scalar) linear transport equations are discussed for instance in \cite{OL, DC08, CD07, santos2010probability} and references therein.
In~\cite{DC11} the authors present expressions for the distribution of the solution of a linear advection equation with a time-dependent velocity, given in terms of the probability density function of the underlying integral of the stochastic process.
Numerical schemes are then introduced in~\cite{DC07,DFC09}.
A stochastic collocation method for the wave equation is introduced in~\cite{motamed2013stochastic}.
Uncertainty quantification of acoustic wave propagation in random heterogeneous layered media is presented in~\cite{mishra2014multi}.
In~\cite{JSK02} the linear advection equation with spatiotemporal coefficients is the subject of research. The authors develop numerical methods using polynomial chaos expansions to solve the advection equation with a transport velocity given by a Gaussian or a log-normal distribution.

\subsection{Aims and contributions of the paper}
We extend the existing framework for random solutions to linear hyperbolic systems presented in, e.g., \cite{vsukys2013multi, mishra2012multilevel,mishra2012sparse,mishra2012multi} and the references therein.
Random coefficients of the numerical fluxes are modeled as \emph{Gaussian random fields in space and Ornstein--Uhlenbeck processes in time}.
Thus, the system of equations consists of a stochastic differential equation (SDE) for \emph{each} random parameter coupled to the conservation law.
For that purpose,
\begin{itemize}
    \item we provide the necessary solution concepts for weak random solutions for linear conservation laws with \emph{time-dependent random field} coefficients, together with a well-posedness result (existence, uniqueness and dependence on initial data).
    \item we describe algorithms for \emph{fast generation} of such spatiotemporal random fields. Discretizations of the SDEs with the appropriate (weak) order are presented.
    \item we present two \emph{applications}: the magnetic induction equation and linear acoustics, both with background velocity fields modeled by spatiotemporal random fields.
\end{itemize}
The simulation results are based on a flexible, thread-parallel algorithm for the uncertainty quantification of linear conservation laws.

The remainder of the paper is organized as follows.
In Section~\ref{sec:theory}, we provide the necessary theoretical background for time-dependent random fields, as well as random (linear) hyperbolic conservation laws.
We present a Monte Carlo Finite Volume framework for approximating moments of the random solutions in Section~\ref{sec:MCFVM}.
This is followed by Section~\ref{sec:examples}, showcasing efficiency and robustness on realistic test cases,
and some conclusions in Section~\ref{sec:conclusion}.

\section{Stochastic linear hyperbolic conservation/balance laws with spatiotemporal random parameters}\label{sec:theory}

This article treats the case of random flux functions with coefficients that are modeled as \emph{spatiotemporal} random fields.
In order to rigorously define uncertainty of the parameters of the flux function, we start by recapitulating the necessary background from probability theory.

\subsection{Spatiotemporal random fields}
To this end, let $(\Omega, \cA)$ be a measurable space with elementary events $\omega\in\Omega$ endowed with a $\sigma$-algebra $\cA$.
Then, the measure space $(\Omega, \cA, \P)$ is called a \emph{probability space}, if $\P:\Omega\rightarrow[0,1]$ is a $\sigma$-additive set function such that $\P(\Omega)=1$.

\begin{definition}[Gaussian random field (GRF)]\label{def:grf}
    A \emph{random field} $g=\{g(\x), \X\in D\}$ (also written as $g(\X,\omega)$) is a set of real-valued random variables defined on a probability space $(\Omega, \cA, \P)$.
    The field $g(\X,\omega)$ is called \emph{Gaussian}, if the vector of random variables follows the multivariate Gaussian distribution for any $\X\in D$, i.e.,
    $g\sim \cN(\mu,C)$, where $\cN$ is the normal distribution with mean $\mu$ and covariance matrix $C(\X,\Y)$.
    Any Gaussian random field is completely defined by its second-order statistics.
\end{definition}

We note that $C$ is a nonnegative, semi-definite, symmetric function.
Bochner's theorem \cite{bochner2012harmonic} states that $C$ is the Fourier transform of a positive measure $\mu_C$ on $\R^d$.
If we assume that $\mu_C$ has a Lebesgue density $\gamma$ which is even and positive, we can construct a GRF by
\begin{equation}\label{eq:grfconst}
    g(\X) = (\F^{-1}\gamma^{1/2}\F W) (\X),
\end{equation}
where $\F$ denotes the d-dimensional Fourier transform with inverse $\F^{-1}$, and $W$ is a centered Gaussian family $W=\{W(\X), \X \in \R^d\}$
with covariance $\E[W(\X)W(\Y)] = \delta(\X-\Y), \X,\Y \in \R^d$ (see~\cite{LP}).

\begin{figure}
\begin{tabular}{lr}
    \subfigure[Ornstein--Uhlenbeck process in time.]{
        \includegraphics[width=0.48\textwidth]{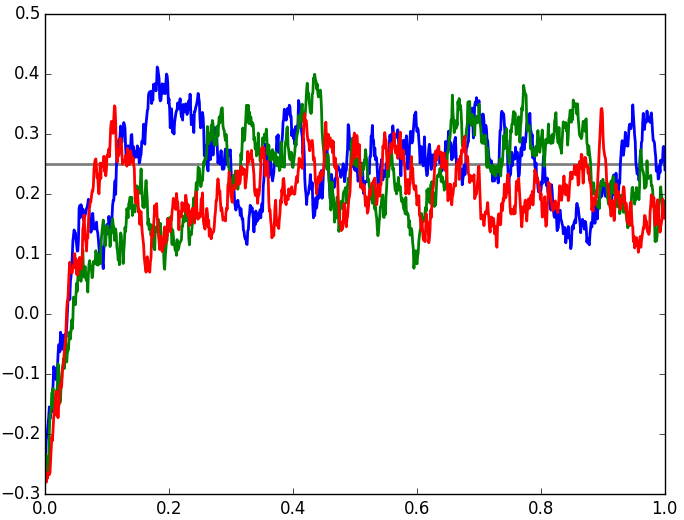}
}
&
    \subfigure[Correlated Gaussian random field in space.]{
        \includegraphics[width=0.48\textwidth]{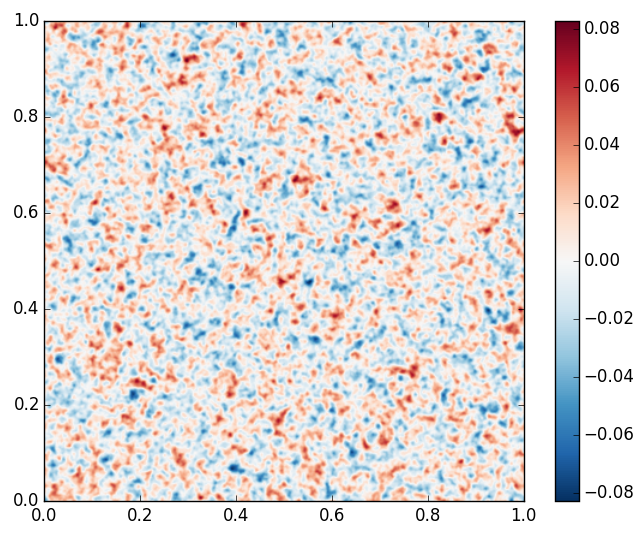}
}
\end{tabular}
    \caption{Sample solutions of random processes in time and random fields in space.}
    \protect \label{fig:examples}
\end{figure}

Looking at the time domain, a standard Brownian motion or Wiener process $B = (B(t), t\in[0,T])$, for $T<\infty$, defined on a probability space $(\Omega, \cA,\P)$, is a continuous stochastic process which starts in zero $\P$-a.s. and has independent and normally distributed increments, i.e., $B_t-B_s\sim\cN(0,t-s)$.
\begin{definition}[Ornstein--Uhlenbeck (OU) process]\label{def:oup}
Let $B=(B(t), t\in[0,T])$ be a standard Brownian motion.
For $\mu,\theta,\sigma\in\R$, $\theta>0$ and $\sigma>0$, the Ornstein--Uhlenbeck process is given as the solution of the stochastic differential equation
\begin{equation}\label{eq:OU}
    \begin{split}
        d a(t) &= \theta(\mu - a(t)) dt + \sigma dB(t),\\
        a(0) &= a_0,
    \end{split}
\end{equation}
In general the initial condition may be random as well.
\end{definition}

The OU processes is \emph{mean-reverting}: Starting at a value $a_0$, over time, the process tends to drift towards its long-term mean $\mu$.
See Figure~\ref{fig:examples} for some sample solutions.
\begin{remark}
    For every $t\in[0,T]$ the random variable $a(t)$ is normally distributed with mean and variance
\begin{align}\label{eq:OU_moments}
    \begin{split}
        \E(a(t)) &= \mu + (a_0-\mu)e^{-\theta t},\\
        \Var(a(t)) &= \frac{\sigma^2}{2\theta}(1-e^{-2\theta t}).
    \end{split}
\end{align}
    This can easily be shown by using It\^o's formula with the function $f(t,x) = \e^{\theta t} x$, and considering the dynamics of $f(t,a(t))$. Then, the stochastic differential equation~\eqref{eq:OU} has the following solution
    \begin{equation}
        \label{eq:OU_direct}
            a(t) = \mu + \e^{-\theta t}(a_0-\mu) + \sigma \int_0^t \e^{-\theta (t-s)}\, dB(s).
    \end{equation}
    From this form we can directly deduce the expectation of $a(t)$, the variance is derived by using the It\^o isometry.
\end{remark}
\begin{remark}
    The realizations of an Ornstein--Uhlenbeck process are continuous and nowhere differentiable with probability 1.
\end{remark}

\begin{definition}[Spatiotemporal random field]
\label{def:timedependentRF}
For all $t\in\R_+$ let $G(t) = \{G(\x,t), \X\in D\}$ be a Gaussian random field with covariance $C$ and mean $0$.
Further, a time-dependent random field is defined as the solution $\Z$ of the following SDE (cf. Definition~\ref{def:oup})
\begin{equation}\label{eq:timedependentRF}
    \begin{split}
        d \Z(\x,t) &= \theta(\mu(\x) - \Z(\x,t)) dt + \sigma dG(\x,t),\\
        \Z(\x,0) &= \Z_0(\x),
    \end{split}
\end{equation}
where $\mu$ is a continuously differentiable function in $L^\infty(D)$, and $\theta,\sigma$ are positive real parameters.

\noindent
The solution $\Z: \R^d \times \R_+ \rightarrow \R$ has the following properties.
\begin{itemize}
    \item For a fixed time $\tilde t\in\R_+$, $\{\Z(\X,\tilde t), \X \in D\}$ is a real-valued Gaussian random field.
    \item For each point $\tilde \x\in \R^d$, $(\Z(\tilde \X,t), t \in \R_+)$ is an Ornstein--Uhlenbeck process, i.e., a mean-reverting process.
\end{itemize}
\end{definition}

\subsection{Random conservation laws}
Equipped with a probability space $(\Omega, \mathcal{F}, \P)$ we can incorporate uncertainties in the flux of Equation~\eqref{eq:general_hyp} by considering the equation
\begin{equation}\label{eq:general_stochhyp}
    \begin{cases}
        \partial_t \U(\x,t,\omega) + \vec\div \cdot (\vec\A(\x,t,\omega) \U(\x,t,\omega)) = \S(\x,t, \omega),\\
        \hfill \U(\x,0, \omega) = \, \U_0(\x, \omega),
    \end{cases} \quad \x \in D\subset\R^d, t>0.
\end{equation}

We are interested in cases, where some or all of the coefficients of $\vec\A$ are modeled as a spatiotemporal random field according to Definition~\ref{def:timedependentRF}.
In order for this to make sense, we follow \cite{vsukys2013multi,mishra2014multi},
but extend the solution concept and definition of hyperbolicity to be time-dependent where necessary.

\begin{definition}[Hyperbolicity]\label{def:hyperbolicity}
For $\w$ in the unit sphere $\mathbb{S}^{d-1}$ let $\breve{\A}^\w (\x,t,\omega) =$\\ $\sum_{r=1}^d \w^{x_r} \A^{x_r}(\x,t,\omega)$ be the convex combinations of the directional random matrices $\A^{x_r}$.
Consider the eigen-decomposition
$$
\breve{\A}^\w (\x,t,\omega) = \Q^\w (\x,t,\omega) \pmb{\Lambda}^\w (\x,t,\omega) \Q^\w (\x,t,\omega)^{-1},
$$
where $\pmb{\Lambda}^\w$ is a diagonal matrix consisting of the eigenvalues $\left(\lambda^\w_r, 1\leq r\leq d\right)$ of $\breve{\A}^\w$, and $\Q^w$ contains the corresponding eigenvectors as columns.
The random linear system of conservation laws \eqref{eq:general_stochhyp} is $\P$-a.s. hyperbolic if all eigenvalues of $\breve{\A}^\w$ are real $\P$-a.s. for all $(\X,t)\in\R^d\times\R_+$.
In addition, for every finite time horizon $T<\infty$ there exists a $K(\omega) < \infty$ such that
\begin{equation}\label{eq:bound}
    \underset{\x\in D, t\in[0,T], \w\in\mathbb{S}^{d-1}}{\sup} \| \Q^\w (\x,t,\omega) \| \|\Q^\w (\x,t,\omega)^{-1} \| \leq K(\omega), \quad \P\text{-a.s.}.
\end{equation}
In addition, we require the \emph{expected wave speeds} to be finite, i.e., 
\begin{equation}\label{eq:eigenvaluebound}
  \widehat\lambda := \underset{1\leq r \leq d}{\max} \ \underset{\x\in D, \w\in\mathbb{S}^{d-1}}{\sup} \left| \E\left[\lambda^\w_r (\X,t,\cdot)\right] \ \right| < \infty, \quad \forall t\in[0,T].
\end{equation}
\end{definition}

We consider a measurable mapping
\begin{equation}
    \U:(\Omega,\mathcal{A}) \rightarrow (V,\mathcal{B}(V)), \quad \omega \mapsto \U(\X,t,\omega),
\end{equation}
where $\mathcal{B}(V)$ is a Borel $\sigma$-algebra of the function space $V$.
Then, we define the concept of a weak solution as follows.
\begin{definition}[Pathwise weak solution]\label{def:randomweaksolution}
    A random field $\U$, with values in $C([0,T],\rL^2(D))$, is called a (pathwise) weak solution to the stochastic conservation law \eqref{eq:general_stochhyp} on $D\times[0,T]$, with $D = \R^d$ and a finite time horizon $T<\infty$, if it is $\P$-a.s. a weak solution.
    This means that it satisfies the variational formulation
    \begin{equation}\label{eq:randomweak}
        \int_{\R^d\times[0,T]} \left(\U\cdot\phibf_t + \sum_{r=1}^d \A^{x_r} \U \cdot \phibf_{x_r} \right) d\X dt + \int_{\R^d} \U_0 \cdot \phibf_{|t=0} d\X = \int_{\R^d\times[0,T]}\S \cdot \phibf \,d\X dt,
    \end{equation}
    for all test functions $\phibf\in C^1_c(\R^d\times[0,T])$ for $\P$-a.e. $\omega \in \Omega$.
\end{definition}

Denote $\rL^p(D) = L^p(D)^d$ and $\W^{r,\infty}(D) = W^{r,\infty}(D)^d$.
We have the following result regarding existence and uniqueness of a solution.
\begin{theorem}[Well-posedness of stochastic linear hyperbolic conservation laws.]\label{theorem:wellposedness}
    Consider the linear conservation law~\eqref{eq:general_hyp_stoch}, and assume the following holds:
    \begin{itemize}
        \item the stochastic differential equation for the coefficients $\V$ (see Equation~\eqref{eq:full_SDE_System}) admits a unique solution (in the sense of~\cite[Chapter 5.2]{KS91})
        \item the system is hyperbolic according to Definition~\ref{def:hyperbolicity}, with (pathwise) constant $\bar K = \|K(\omega)\|_{\rL^k(\Omega,\R)} < \infty$,
        \item the moments of the initial data, the source and the flux are bounded in the following sense: there exist non-negative $r_0,r_S,r_A\in\N\cup\{0,\infty\}$ such that
            \begin{equation}
                \U_0 \in \rL^k(\Omega, \W^{r_0,\infty}), \quad 
                \S \in \rL^k(\Omega, \W^{r_S,\infty}), \quad 
                \A^{x_r} \in \rL^k(\Omega, \W^{r_A,\infty}), \quad 
            \end{equation}
            where $\Omega$ is the sample space of the probability space.
        \item each random field $\A^{x_r}$ is stochastically independent of $\U_0$ and $\S$.
    \end{itemize}
    Then, for $T<\infty$, the system~\eqref{eq:general_hyp_stoch} admits a unique pathwise weak solution.
    Moreover, for all $t\in[0,T]$, we have the following estimates
    \begin{align}
        \begin{split}
            \label{eq:estimates}
            \|\U(.,t,\omega)\|_{\rL^2(D)} &\leq K(\omega)\left( \| \U_0(.,\omega)\|_{\rL^2(D)} + t \|\S(.) \|_{\rL^2(D)} \right), \quad \P\text{-a.s.},\\
            \|\U\|_{\rL^k(\Omega,C([0,T],\rL^2(D))} &\leq \bar K \left( \| \U_0\|_{\rL^k(\Omega,\rL^2(D))} + t \|\S \|_{\rL^k(\Omega,\rL^2(D))} \right),
        \end{split}
    \end{align}
\end{theorem}

\begin{proof}
    The proof is an immediate consequence from~\cite[Theorem~1]{vsukys2013multi}, if one considers the following modifications to address the time-dependence of the flux function:  
    \begin{enumerate}
        \item Using the time-dependent version of the classical existence and uniqueness results summarized in~\cite[Theorem~1]{vsukys2013multi}, one can show that the random field given by the $\omega \mapsto \U(\cdot,\cdot,\omega)$ is well defined and that $U(\cdot,\cdot,\omega)$ is a weak solution $\P$-a.s..
        \item To show that the maps $\omega\mapsto \U(\cdot,t,\omega)$ are measurable for all $t\in[0,T]$ $\P$-a.s. one uses the fact that $L^2(D)$ is a separable Hilbert space and the stability estimate of the classical theorem for the deterministic (pathwise) solution.
        \item The first estimate in~\eqref{eq:estimates} may again be derived from the time-dependent version of the classical case (summarized in~\cite[Theorem~1]{vsukys2013multi}), and the second follows from the first, using the assumption that random fields, initial conditions and sources are independent.
    \end{enumerate}
    Incorporating the above into the proof of~\cite[Theorem~1]{vsukys2013multi}, the assertion follows immediately.
\end{proof}

If $\U_0,\S\in L^k(\Omega,\rL^2(D))$ and $K\in L^k(\Omega,\R)$ Theorem~\ref{theorem:wellposedness} ensures the existence of moments of order $k$ of the random weak solution.

In general, there are no explicit solution formulas.
For the special case of the scalar linear transport equation with a time-dependent coefficient (transport driven by the Ornstein--Uhlenbeck process) however, we can derive the distribution of the solution in closed form.
This distribution will then be used in an example presented in Section~\ref{sec:OUlinear} to verify our Monte Carlo Finite Volume method.
\begin{theorem}\label{theorem_lineartime}
    Consider the scalar transport equation
    \begin{equation}\label{eq:scalarlinearOU}
        \begin{cases}
            u(x,t,\omega)_t + (a(t,\omega) u(x,t,\omega))_x = 0\\
            u(x,0) = u_0(x)
        \end{cases}
    \end{equation}
    with coefficient $a(t,\cdot)$ given by the Ornstein--Uhlenbeck process \eqref{eq:OU}.
    The moments of the solution to Equation~\eqref{eq:scalarlinearOU} exist and are given by
    \begin{equation}\label{eq:time_coeff_moment0}
            \E(u(x,t)) = \int u_0(x-y)f_{A(\hat{\sigma}^2,\hat{\mu})}(y) \,dy = (f_{A(\hat{\sigma}^2,0)}*u_0)(x-t\hat{\mu}).
    \end{equation}
    Here, "$*$" denotes convolution and the probability density function $f_A$ is given by
    \begin{equation*}
        f_{A(\hat{\sigma}^2,\hat{\mu})}(y) = \frac{1}{\sqrt{2\pi \hat{\sigma}^2}} e^{-\frac{(y-\hat{\mu})^2}{2\hat{\sigma}^2}},
    \end{equation*}
    with diffusion coefficient $\hat{\sigma}^2=\frac{\sigma^2}{\theta^3} \big(\theta t + 2\e^{-\theta t} - \frac{1}{2}\e^{-2\theta t} - \frac{3}{2}\big)$ and transportation speed $\hat{\mu}=\mu - (a_0-\mu)\frac{\e^{-\theta t}-1}{\theta t}$.
\end{theorem}
\begin{remark}
    Higher moments of the solution may be calculated by
    \begin{align}
            \M_m(u(x,t)) = \E\big((u(x,t)-\E(u(x,t)))^m\big).
    \end{align}
\end{remark}

\begin{proof}
    The solution for a single realization (fixed $\omega\in\Omega$) of Equation~\eqref{eq:scalarlinearOU} is given by $u_0(x-\int_0^ta(s,\omega)\,ds)$. 
    We start by calculating the first moment of this expression, i.e.
    \begin{equation*}
        \E(u_0(x-\int_0^ta(s)\,ds)).
    \end{equation*}
    This means, we have to calculate the distribution of the time integral over $a$, i.e. the distribution of the stochastic process
    \begin{equation*}
        A(t) = \int_0^t a(t)\, dt.
    \end{equation*}
    The process $A$ is again a Gaussian process, i.e. $A(t)\sim\cN(\hat{\mu}, \hat\sigma^2)$, and therefore completely characterized by its mean and variance. Using Fubini's theorem we have that
    \begin{align}
        \label{eq:OUtime_mu}
        \begin{split}
            \E(A(t))
            = \int_0^t \E(a(s))\, ds
            = \int_0^t \mu + \e^{-\theta s}(a_0-\mu)\, ds
            = \mu t - (a_0-\mu)\frac{\e^{-\theta t}-1}{\theta}
            =:\hat{\mu}.
        \end{split}
    \end{align}
    We express the variance of $A$ via the covariance of $A$ with itself
    \begin{equation*}
        \var(A(t)) = \text{Cov}(A(t),A(t)) = \E\big((A(t)-\E(A(t)))(A(t)-\E(A(t)))\big).
    \end{equation*}
    Using $A(t)-\E(A(t))=\sigma\int_0^t\int_0^s\e^{-\theta(s-u)}\,dB(u)\,ds$ (combine Equations~\eqref{eq:OU_direct} and \eqref{eq:OUtime_mu}) this yields
    \begin{align*}
        \var(A(t)) &= \E\big(\sigma \int_0^t \int_0^s \e^{-\theta(s-u)}\,dB(u)\,ds\, \sigma \int_0^t \int_0^r \e^{-\theta(r-v)}\,dB(v)\,dr\big)\\
        &= 2\sigma^2 \int_0^t \e^{-\theta s}\int_0^t\e^{-\theta r} \,\E\big( \int_0^s \e^{\theta u}\,dB(u) \int_0^r \e^{\theta v}\,dB(v)\big)\,dr\,ds,
    \end{align*}
    using Fubini's theorem.
    For a Brownian motion $B$, it is known that $$\E( \int_0^s \e^{\theta u}\,dB(u) \int_0^r \e^{\theta v}\,dB(v))= \frac1{2\theta}(e^{2\theta \text{min}(s,r)}-1).$$
    Therefore, we have
    \begin{align*}
        \var(A(t)) &= 2\sigma^2 \int_0^t \e^{-\theta s}\int_0^s\e^{-\theta r} \frac1{2\theta}(e^{2\theta \text{min}(s,r)}-1)\,dr\,ds\\
        &= \frac{\sigma^2}{\theta} \int_0^t \e^{-\theta s}\int_0^s\e^{-\theta r}(e^{2\theta r}-1)\,dr\,ds
        = \frac{\sigma^2}{\theta^3} \big(\theta t + 2\e^{-\theta t} - \frac{1}{2}\e^{-2\theta t} - \frac{3}{2}\big)
        =:\hat{\sigma}^2.
    \end{align*}
    This gives us the variance of $A(t)$ depending on the variables $\theta$ and $\sigma$.

    Therefore, the expectation of the solution to Equation~\eqref{eq:scalarlinearOU} is given by
    \begin{equation*}
        \E(u_0(x-\int_0^ta(s)\,ds)) = \E(u_0(x-A(t)))
        =\int_{-\infty}^\infty u_0(x-y) f_A(y) \,dy
    \end{equation*}
    where $f_A$ is the normal density function with parameters $\hat{\mu}$ and $\hat{\sigma}^2$ given by
    $
        f_A(y) = \frac1{\sqrt{2\pi\hat{\sigma}^2}}\e^{-\frac{(y-\hat{\mu})^2}{2\hat{\sigma}^2}}.
    $
\end{proof}

\medskip
\begin{remark}
    For the limit $\theta \rightarrow 0$, we recover the corresponding result for a pure Brownian motion process (i.e. $a(t)= \sigma B(t)$), where $\hat{\mu} = \mu$ and $\hat{\sigma}^2 = \sigma^2\frac{t^3}{3}$.
    This can be shown by a Taylor expansion as
    \begin{align*}
        \var(A(t)) &= \frac{\sigma^2}{\theta^3} \big(\theta t + 2\e^{-\theta t} - \frac{1}{2}\e^{-2\theta t} - \frac{3}{2}\big)\\
                   &= \frac{\sigma^2}{\theta^2} \Big(\theta t + 2\left(1 -\theta t +\theta^2 t^2/2 -\theta^3 t^3/3! + \Ord(\theta^4)\right)\\
                   &\quad \quad \quad - \frac{1}{2}\left(1 -2\theta t +4\theta^2 t^2/2 - 2^3\theta^3 t^3/3! + \Ord(\theta^4)\right)- \frac{3}{2}\Big)
                   = \sigma^2 t^3/3 + \Ord(\theta).
    \end{align*}
    A similar Taylor expansion shows the result for $\E(A(t))$.
\end{remark}

\section{Monte Carlo Finite Volume methods}\label{sec:MCFVM}

For the approximation of the (moments of the) solution to partial differential equations with random coefficients we have to discretize in space and time, as well as in the ``stochastic domain".
We use a Monte Carlo method for the approximation of moments of the random solution.
This means
that we have to approximate the (deterministic) solution for each realization $\omega$ of Equation~\eqref{eq:general_stochhyp}, i.e.,
\begin{align}\label{eq:full_SDE_System}
    \begin{split}
        \partial_t \U(\x,t) + \vec\div \cdot \left(\vec\A(\x,t,\V(\X,t)) \ \U(\x,t)\right) &= \S(\x,t),\\
        \partial_t \V(\X,t) &= \boldsymbol{\mu}(\V(\X,t)) + \boldsymbol{\sigma}(\V(\X,t)) \ \boldsymbol{\nu}(\X,t), \\
    \end{split}
\end{align}
where each component of $\V$ is given by the spatiotemporal random field~\eqref{eq:timedependentRF}.
As mentioned in the Introduction, a Multilevel MC method could be used to achieve optimal computational complexity.
This is, however, not the goal of this article.

Before we start with a brief recapitulation of Finite Volume methods, we introduce some useful notation.
Let the computational domain be a bounded axiparallel domain, i.e., $D=D_x\times D_y\times D_z$.
A uniform axiparallel mesh of the domain $D$ consists of identical cells $C_{i,j,k}$, for $1\leq i\leq I$, $1\leq j\leq J$, $1\leq k\leq K$ and $I,J,K<\infty$, with edge lengths $\Delta x, \Delta y, \Delta z$ in the x-, y- and z-direction, respectively.
The cell centers are denoted by $\x_{i,j,k}$, and values at the cell interfaces by $x_{i-1/2}, y_{j-1/2}, z_{k-1/2}$.
For a function $f(\X,t):D\rightarrow \R$ we set
$
f_{i,j,k}^n = f(\X_{i,j,k}, t^n),
$
where $t^n$, for $n=1,\ldots,N$, is the n-th time step.

\subsection{Finite volume methods}

A Finite Volume (FV) scheme is obtained by integrating Equation~\eqref{eq:general_hyp} over a cell, or control volume, $C_{i,j,k}$, and over a time interval $T^n=[t^{n},t^{n+1}]$, $t^{n+1}=t^n + \Dt$.
Denoting cell averages by $\U_{i,j,k}(t) = \frac{1}{|C_{i,j,k}|}\int_{C_{i,j,k}} \U(\X,t) \,d\x$, a fully discrete flux-differencing method has the form
\begin{align}\label{eq:FVtime}
    \begin{split}
        \U_{i,j,k}^{n+1} = \U_{i,j,k}^n &- \frac{\Dt}{\Dx}(F^n_{i+1/2,j,k} - F^n_{i-1/2,j,k})\\
                          & - \frac{\Dt}{\Dy}(G^n_{i,j+1/2,k} - G^n_{i,j-1/2,k}) - \frac{\Dt}{\Dz}(H^n_{i,j,k+1/2} - H^n_{i,j,k-1/2}),
    \end{split}
\end{align}
where $F,G,H$ are the fluxes in x-, y- and z-direction respectively.
The fluxes $F^n_{i+1/2,j,k}$ approximate the integral
\begin{equation} \label{eq:fluxapprox}
    F^n_{i+1/2,j,k} \approx \frac{1}{\Dt \Dy \Dz}\int_{T^n}\int_{y_{j-1/2}}^{y_{j+1/2}}\int_{z_{k-1/2}}^{z_{k+1/2}}f(x_{i+1/2}, y, z) dt \, dy \, dz.
\end{equation}
The approximations for the fluxes in the other directions are equivalently defined.
Using the flux-differencing form, FV methods are typically based on the reconstruct-evolve-average (REA) algorithm.
This algorithm consists of the following steps performed for each time step:
First, one reconstructs the cell averages by a piecewise polynomial function.
Then, one evolves the hyperbolic equation by defining the numerical fluxes~\eqref{eq:fluxapprox}.
Finally, the new cell averages are computed.
The numerical fluxes are based on solutions of local Riemann problems at each cell interface.
High order accuracy in space is achieved by using TVD limiters, or (W)ENO schemes \cite{harten1987uniformly,LEV1,VL1,HEOC1,SO1},
and in time by strong stability-preserving (SSP) Runge-Kutta methods, see e.g., \cite{gottlieb2001strong}.
The CFL condition dictates a limit on the time step size $\Delta t^n$ for the resulting explicit FV schemes.
Let $\lambda_p^x,\lambda_p^y,\lambda_p^z$ be the eigenvalues of Equation~\eqref{eq:general_hyp} in x-, y-, and z-direction, then the time step size $\Dt$ needs to satisfy
\begin{equation}\label{eq:CFL}
    \frac{\Delta t^n}{\Delta x} \overline{\lambda^x} 
    +
    \frac{\Delta t^n}{\Delta y} \overline{\lambda^y} 
    +
    \frac{\Delta t^n}{\Delta z} \overline{\lambda^z} 
    \leq \frac{1}{2},
\end{equation}
where $\overline{\lambda^x} = \underset{i,j,k}\max \, \underset{p}\max |\lambda^x_p(\X_{i,j,k},t^n)|$ is the largest \emph{absolute} eigenvalue in x-direction, and similar in y-, and z-directions.

\subsection{Realizations of spatiotemporal random fields}\label{sec:RFapprox}
In order to approximate moments of the solution to Equation~\eqref{eq:general_stochhyp}, we need to create realizations of the spatiotemporal random fields given in Definition~\ref{def:timedependentRF}.
Therefore, we describe here one way to discretize the SDE~\eqref{eq:timedependentRF} in time and space.
\subsubsection{Discretization in time}
There is a wide variety of numerical methods for SDEs.
Methods are classified mainly according to their \emph{strong} ($p_s$) and \emph{weak} ($p_w$) orders of convergence.
The order of weak convergence is typically higher than the order of strong convergence of a scheme.
For example the Euler-Maruyama approximation has convergence orders $(p_s, p_w) =(1/2,1)$,
and the Milstein method has convergence orders $(p_s,p_w) = (1,1)$.
For a classification of higher order SRK schemes see, e.g., \cite{debrabant2008classification}.

For the Ornstein--Uhlenbeck process, given in Equation~\eqref{eq:timedependentRF}, the Milstein method is equivalent to the Euler-Maruyama method, since $\boldsymbol{\sigma}(\V)=\boldsymbol{\sigma}$ is independent of the process.
The Milstein method is given by 
\begin{equation}\label{eq:OU_discrete_original}
    Z^{l+1}(\X) - Z^l(\X) = h \theta\left(\mu(\X) - Z^{l}(\X)\right) + \sigma \sqrt{h} G^l(\X),
\end{equation}
where $Z^l(\X) = Z(\X, t^l)$ for all $\X\in D$ and $G^l(\X) = G(\X,t^l)$ is a Gaussian random field in $\X$ at the discrete times $t^l$.

Depending on the parameters, the Ornstein--Uhlenbeck process, given in Equation~\eqref{eq:timedependentRF}, is \emph{stiff} and appropriate \emph{implicit} schemes have to be used.
For instance, the implicit Milstein method is given by
\begin{equation}\label{eq:OU_discrete}
    \begin{cases}
        Z^{l+1}(\X) &= \left(Z^l(\X) + h\theta\mu(\X) + \sigma\sqrt{h} G^l(\X)\right)/\left(1 + h\theta\right), \\
        Z^0 &= \mu(\X),
    \end{cases} \quad \forall \X \in D.
\end{equation}

In general, higher order SRK schemes become increasingly involved, for the Ornstein--Uhlenbeck process, however, those schemes simplify considerably, as the function $\boldsymbol{\sigma}(\V)=\boldsymbol{\sigma}$ is constant.

\subsubsection{Discretization in space}
The discretization in Equation~\eqref{eq:OU_discrete} is only semi-discrete as it is continuous in space.
For a fully discrete approximation of a realization of Equation~\eqref{eq:timedependentRF} we need to provide an algorithm that approximates realizations of the Gaussian random fields $G^l(\X)$ on a Cartesian grid over $D\subset\R^d$ at each time step.
To this end, we use the approach described in \cite{LP}, which provides the following algorithm based on the representation formula \eqref{eq:grfconst}.
For simplicity we only present the periodic case.
Let $\mathcal{F}$ be the discrete Fourier transform with inverse $\mathcal{F}^{-1}$, and let $Z_{i,j,k}$ be random samples from a normal distribution for all $0\leq i<I ,0\leq j<J,0\leq k<K$, i.e., $Z_{i,j,k} \sim \mathcal{N}(0,|\Delta_c|^{-1})$, where $|\Delta_c|$ is the volume of the cells. Then, a Gaussian random field is given by
\begin{equation}\label{eq:GRF_approx}
        \{G^l_{i,j,k}\} = \mathcal{F}^{-1}\left\{\sqrt{\gamma_{i,j,k}}\,\mathcal{F}\left\{Z_{i,j,k}\right\}\right\},
\end{equation}
where the covariance is given by the Fourier transform of the symmetric, positive function $\gamma$.
A typical family of functions for the Lebesgue density $\gamma$ is given by
\begin{equation}\label{eq:correlation}
    \gamma_{i,j,k} = (1 + \|\mathbf{p}_{i,j,k}\|_2^{q})^{-l}, \quad k,l\in\N,\, q,l\geq1,
\end{equation}
where the points $\mathbf{p}_{i,j,k}$ in the Fourier domain are given by $\mathbf{p}_{i,j,k} = (i-I/2, j-J/2, k-K/2)^T$.
Another possibility would be to employ an exponential covariance function given by
    $
    \gamma_{i,j,k} = \exp(-\|\mathbf{p}_{i,j,k}\|_2/v) , 
    $
with correlation length $v$.

\subsection{Monte Carlo approach}

We employ a Monte Carlo Finite Volume method (MC FV method) in order to approximate the moments of the stochastic PDE \eqref{eq:general_hyp_stoch} where particular coefficients are modeled as spatiotemporal random fields given in Equation~\eqref{eq:timedependentRF}, such that the system is hyperbolic according to Definition~\ref{def:hyperbolicity}. As before, we consider a bounded axiparallel domain $D$ together with a mesh $\mathcal{T}$ consisting of identical cells $C_{i,j,k}$.

The step size $\Dt$ of the explicit PDE solver is given by the CFL condition~\eqref{eq:CFL} and depends on the eigenvalues which are influenced by the stochastic parameters $\V$.
On the other hand, the interval length $h$ of the discretization of the SDE is constant for each realization.
To achieve an optimal convergence rate, the approximation error of the SDE's and the FV method should be of the same order.
If a bound of the maximum expected eigenvalue $\widehat\lambda$, as defined in Equation~\eqref{eq:eigenvaluebound}, is available, one can choose
\begin{equation}
    h = c \, \frac{\Delta x}{\widehat\lambda},
\end{equation}
where $c>0$ is a constant.
We require the discrete times $t^n$ of the FV scheme to be a subset of the discrete times of SDE approximation, i.e., $\{t_n, \, n=1,\ldots,N\} \subset \{t^l = l \cdot h, \, l=1,\ldots,L\}$.

The MC FV algorithm consists of three main steps.
\begin{enumerate}

    \item Generate $M$ realizations $(Z(x,t,\omega_m),\,1\leq m\leq M)$ \emph{for each parameter} that is modeled as a spatiotemporal random field~\eqref{eq:timedependentRF}:
        \begin{itemize}
            \item Choose an appropriate interval $h = T/L$ for the approximation of the SDE.
            \item Generate $(L-1) \cdot M$ (times the number of parameters) Gaussian random fields $G^l_{i,j,k}$ on the given mesh $\mathcal{T}$ according to Equation~\eqref{eq:GRF_approx}.
            \item Use an approximation scheme with the same (weak) order as the FV method.
        \end{itemize}
    \item For each generated realization, the deterministic problem of the underlying hyperbolic conservation law is solved on the given mesh $\mathcal{T}$.
        In the underlying Riemann problem, the random field $Z^{l}_{i,j,k}$ is assumed to be piecewise constant on the time intervals $[lh, (l+1)h]$.
        For the sample $\omega_m$ we denote by $\U(\x,T,\omega_m)$ the exact pathwise solution at time $T$, and by $\U_\mathcal{T}(\x,T,\omega_m)$ the numerical approximation.

    \item The $M$ approximations of the sample solutions, i.e., $\U_\mathcal{T}(\x,T,\omega_m)$ are used to approximate moments of the random solution field $\U(\x,T,\omega)$.
\end{enumerate}

One is particularly interested in the first two moments, i.e., the expectation $\E[\U]$ and the variance $\Var[\U]$.
The sample mean of the approximate solutions is used to estimate the expectation given by
\begin{equation}
    E_M[\U_\mathcal{T}] (\x,T)= \frac{1}{M}\sum_{m=1}^M \U_\mathcal{T}(\x,T,\omega_m).
\end{equation}
Higher statistical moments of $\U$, such as the variance $V_M[\U_\mathcal{T}] (\x,T)$, can be approximated similarly.
The total approximation error of the expectation in the $L^1$ norm, i.e.
\begin{equation}\label{eq:E_appr}
    \eps_{\text{appr}}(T) = |C_{i,j,k}| \ \sum_{i,j,k} \Big| E_M[\U_\mathcal{T}](\X_{i,j,k},T) - \E[\U](\x_{i,j,k},T) \Big|,
\end{equation}
is bounded by the sum of the numerical approximation error $\eps_{\text{num}}$ of the base method and the Monte Carlo error $\eps_{\text{MCM}}$
\begin{equation}\label{eq:errorest}
    \eps_{\text{appr}}(t) \leq \eps_{\text{num}}(t) + \eps_{\text{MCM}}(t),
\end{equation}
where
\begin{align}
    \begin{split}
        \eps_{\text{num}}(t) &= |C_{i,j,k}| \ \sum_{i,j,k} \Big| \frac{1}{M} \sum_{m=1}^M \left( 
\U_\mathcal{T}(\x_{i,j,k},T,\omega_m) - \U(\x_{i,j,k},T,\omega_m)
\right) \Big|
, \\
        \eps_{\text{MCM}}(t) &= |C_{i,j,k}| \ \sum_{i,j,k} \Big| \frac{1}{M} \sum_{m=1}^M 
\U(\x_{i,j,k},T,\omega_m) - \E[\U](\x_{i,j,k},T)
 \Big|
.
    \end{split}
\end{align}
Using the triangle inequality, it is trivial to show the relationship \eqref{eq:errorest}. If one uses the $L^2$ norm then equality holds.

The estimate~\eqref{eq:errorest} shows that the approximation error is bounded by the dominating part of the sum of the numerical error and the Monte Carlo error. The Monte Carlo method converges with the rate $1/2$ in the number of samples in mean square and is independent of the resolution of the grid, i.e. the size of $\Dx$. On the other hand, a numerical base method of order $o$ converges with $\mathcal{O}(\Dx^o)$ for each single realization, independent of the number of Monte Carlo samples.
Therefore, equation~\eqref{eq:errorest} suggests that our Monte Carlo method is most efficient if $\eps_{\text{num}} \simeq \eps_{\text{MCM}}$. For the according sample numbers in a MLMC approach we refer to~\cite{BL12}. 
In \cite{mishra2014multi} for instance it is shown that this can be achieved if the number of Monte Carlo samples is $M = \mathcal{O}(\Delta x^{-2o})$, where $o$ is the order of the FV method.

\begin{lemma}
    Assume that,
    \begin{itemize}
        \item the conditions of Theorem~\ref{theorem:wellposedness} are fulfilled for $k\geq2$, 
        \item the underlying numerical FV scheme converges (under grid refinement) to the weak solution of Equation~\eqref{eq:general_hyp} with rate $o>0$,
        \item the numerical scheme for the SDE converges at the same rate $o>0$.
    \end{itemize}
    Then, the MC FV method estimates $E_M[\U_\mathcal{T}]$ described in this section converge to the first moment of the solution $\E[\U]$ in mean square sense, as $\Delta x \rightarrow 0$, with $M = \mathcal{O}(\Delta x^{-2o})$.
\end{lemma}

\begin{proof}
 The assertion follows directly from the triangle inequality and the convergence of the corresponding FV scheme, together with the specific choice for the number of samples
 \begin{align*}
  \E\|\E[\U]-E_M[\U_\mathcal{T}]\|^2 &\leq \|\E[\U]-\E[\U_\mathcal{T}]\|^2 + \E\|E_M[\U_\mathcal{T}]-\E[\U_\mathcal{T}]\|^2\\
   &\leq \E\|\U-\U_\mathcal{T}\|^2 + \frac1{M}\text{Var}(\U_\mathcal{T})\\
   &= \mathcal{O}(\Delta x^{2o}).
 \end{align*}
 Here we used the standard convergence properties of the sequence of Monte Carlo estimators $(E_M, M\in\N)$. The norm $\|\cdot\|$ is the canonical norm for the (pathwise) solution $\U$. 
\end{proof}

There is no dependence between different samples $\omega_m$, and therefore the described algorithm is trivial to parallelize. Our implementation distributes the workload on as many CPU threads as are available in the computing environment, achieving (trivially) optimal parallelization-efficiency.

\section{Examples}\label{sec:examples}

We evaluate the proposed approach for uncertainty quantification of linear hyperbolic conservation laws on a suite of test cases.
We start with a ``degenerate'' case, where an autonomous scalar linear transport is driven by a (space-\emph{in}dependent) Ornstein--Uhlenbeck process. We present error and convergence analysis, based on the existence of the explicit solution formula~\eqref{eq:time_coeff_moment0}.
\begin{figure}[t]
    \centering 
        \includegraphics[width=0.5\textwidth]{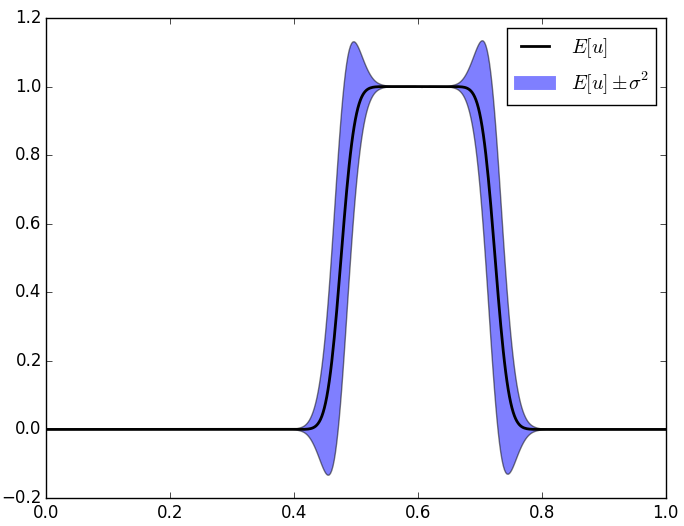}
        \caption{Expectation and standard deviation of the solution of the OU-driven linear transport~\eqref{eq:linear1dOU}.}
    \protect \label{fig:linear1dOU}
\end{figure}
Then, we present two realistic cases of linear systems of conservation laws in two dimensions, namely the equations for linear acoustics and the motion of magnetic fields (induction equation).
In both cases, we model the (background) velocity field in x-, and y-direction as a spatiotemporal random field, and we show hyperbolicity of the system.
The simulations show the robustness of the approach and reveal interesting features of the moments of the solutions.

\subsection{Ornstein--Uhlenbeck process driven scalar linear advection}\label{sec:OUlinear}
We \\start by considering the scalar linear stochastic conservation law
with a parameter given by the Ornstein--Uhlenbeck process \eqref{eq:OU}, i.e.,
\begin{align}\label{eq:linear1dOU}
 \begin{split}
        u_t + (a(t) u)_x &= 0, \\
        u(x,0,\omega) &= u_0(x), \quad x\in D = [0,1]\\
        d a(t)           &= \theta(\mu - a(t)) dt + \sigma dB(t),\\
	a(0) &= a_0,
 \end{split}
\end{align}
with periodic boundary conditions for $u$.
The eigenvalue of the PDE is the matrix itself, i.e., $\lambda(x,t,\omega) = a(t,\omega)$, and the normalized eigenvector is 1.
The system is hyperbolic according to Definition~\ref{def:hyperbolicity}, since $\|Q\|\|Q^{-1}\| = 1$ and, using Equation~\eqref{eq:OU_moments}, we have that the expected maximum eigenvalue
\begin{equation}
    \left|\E[\lambda(x,t,\omega)] \right| = \left|\E[a(t,\omega)] \right| = \left|\mu + (a_0-\mu)e^{-\theta t} \right| < \infty, \, \text{ for all } t\in\R_+
\end{equation}
is finite.

\begin{figure}
    \centering
    \begin{tabular}{lr}
        \subfigure[{Relative $L^2$-error of the mean.}]{
            \includegraphics[width=0.48\textwidth]{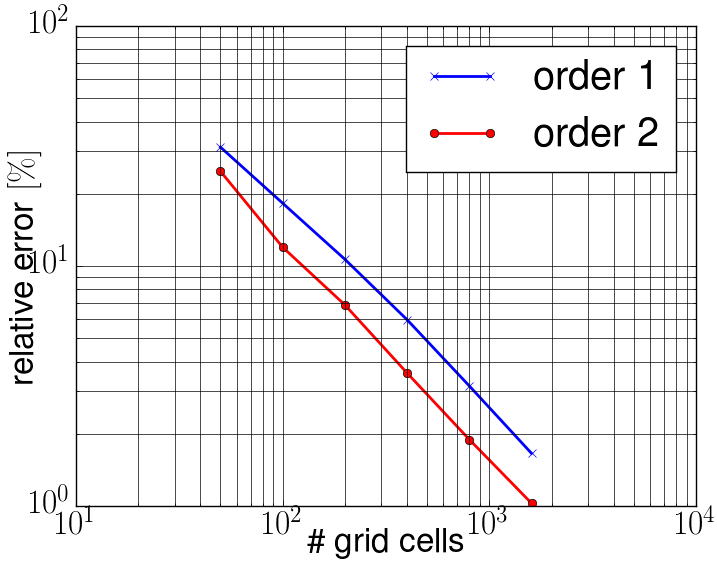}
        }
        &
        \subfigure[{Relative $L^2$-error of the variance.}]{
            \includegraphics[width=0.48\textwidth]{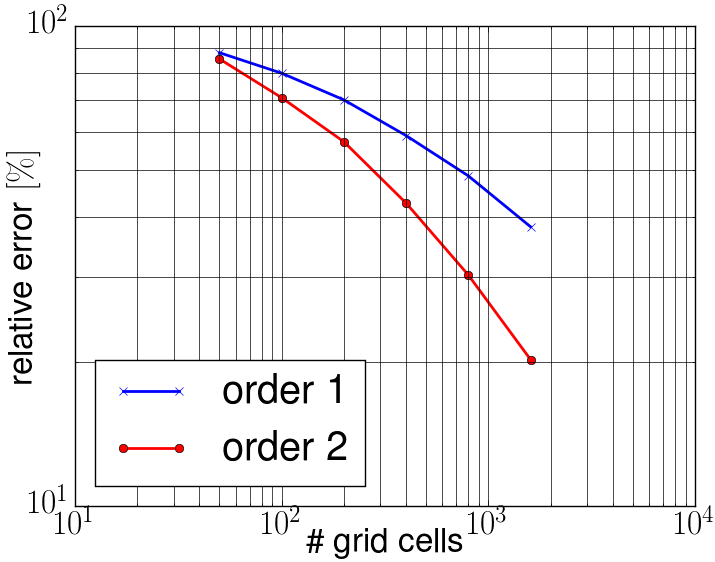}
        }
    \end{tabular}
    \caption{
        Relative errors (in \%) at time $t=1$ for the OU-process-driven linear transport Equation~\eqref{eq:linear1dOU}.
Both the first and second order MC FV method converge to the exact solution.
    }
    \protect \label{fig:convergencelinear1dOU}
\end{figure}

We use the derived explicit solution formula from Theorem~\ref{theorem_lineartime} for the moments for the solution to this equation to show convergence of the proposed MC FV method.
To this end, we compute the first two moments of Equation~\eqref{eq:linear1dOU} at time $t=1$, started with an initial condition consisting of a discontinuity, i.e. $u_0(x) = {\bf 1}_{[\frac{1}{2}-\frac{1}{8}, \frac{1}{2}+\frac{1}{8}]}$, where ${\bf 1}_A$ is the characteristic function of the subset $A$.
Furthermore, we choose the deterministic initial condition of the OU process to be $a(0) = -\frac{1}{4}$ and $(\mu, \theta, \sigma) = (\frac{1}{4}, 20, \frac{1}{2})$.
A few sample solutions for these parameters are plotted in Figure~\ref{fig:examples}(a).
We can see in Figure~\ref{fig:linear1dOU} that the expectation $\E(u)$ at time $t$ consists of the initial function $u_0$ transported with speed $\hat{\mu}$ and smoothed out wave fronts in accordance with Theorem~\ref{theorem_lineartime}. The largest values of the variance are located around the (smoothed out) discontinuities.

The first order scheme MC FV scheme uses a standard upwind discretization of the deterministic problem, and the Milstein scheme for the OU--process.
The second order scheme consists of a minmod flux-limiter in space and a second order strong stability preserving (SSP)  Runge-Kutta time-stepping for the deterministic problem, together with a (weak) second order stochastic Runge-Kutta scheme for the OU--process.
For the discrete time interval of the OU--process we use $h = \frac{\Dx}{2 \widehat\lambda} = 2\Dx$.
The number of Monte Carlo samples is chosen as $M = \Ord(\Delta x^{-2o})$.
For the described setup, Figure~\ref{fig:convergencelinear1dOU} shows the relative approximation error of the first two moments of the solution.
Both the first and second order MC FV method converge to the exact solution.
The convergence order for the first moment is $s\approx1$, and for the second moment is $s\approx0.7$.
The second order MC FV method has a smaller error constant compared to the first order method.
The full convergence order $s=2$ for the second order scheme is not achieved, since the deterministic solution consists of discontinuous piecewise linear data.

\subsection{Linear Acoustics in 2 dimensions}
Sound waves can be described using the vector of conserved variables
$\U = (p,u,v)^T$, where $p$ is the pressure, $u$ is the velocity in x-direction, and $v$ in y-direction.
Given a background density $\rho_0 \in \R_+$ and a bulk modulus of compressibility $K_0 \in \R_+$, 
the dynamics are governed by the following system of equations
\begin{align}
    \begin{split}
    \U_t 
    +&
    \left(
        \breve{\A}^{\left(
                \begin{smallmatrix}
                    1\\
                    0
                \end{smallmatrix}
            \right)
        }
        \U \right)_x
    +
    \left(
        \breve{\A}^{\left(
                \begin{smallmatrix}
                    0\\
                    1
                \end{smallmatrix}
            \right)
        }
        \U \right)_y
    = 0, \\
    \breve{\A}^\w = \breve{\A}^\w (\x,t,\omega) = &
    \begin{pmatrix}
        \breve{u}^\w_0(x,t,\omega) & \text{w}_x K_0 & \text{w}_y K_0 \\
        \text{w}_x /\rho_0 & \breve{u}^\w_0(x,t,\omega) & 0 \\
        \text{w}_y /\rho_0 & 0 & \breve{u}^\w_0(x,t,\omega)
    \end{pmatrix},\\
    \breve{u}^\w_0 &= u_0 \text{w}_x + v_0 \text{w}_y,
    \end{split}
    \protect 
    \label{eq:linear_acoustics}
\end{align}
where
$u_0(\X,t,\cdot)$ is the stochastic background velocity field in x-direction,
and $v_0(\X,t,\cdot)$ in y-direction.
Both $u_0$ and $v_0$ are spatiotemporal random fields as specified in Definition~\ref{def:timedependentRF}, i.e., they are given as the solution of the following SDEs
\begin{align}\label{eq:linearAccVelField}
    \begin{split}
         d u_0(\x,t) &= \theta_u(\mu_u(\x) - u_0(\x,t)) dt + \sigma_u dG_u(\x,t),\\
         d v_0(\x,t) &= \theta_v(\mu_v(\x) -\, v_0(\x,t)) dt + \sigma_v dG_v(\x,t).
    \end{split}
\end{align}

\begin{figure}
\centering
\begin{tabular}{lcr}
    \subfigure[Pressure.]{
\includegraphics[width=0.31\textwidth]{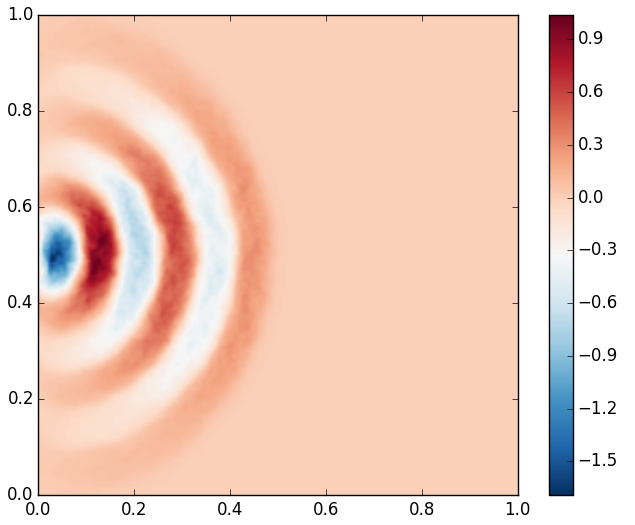}
}
&
    \subfigure[Velocity in x-direction.]{
\includegraphics[width=0.31\textwidth]{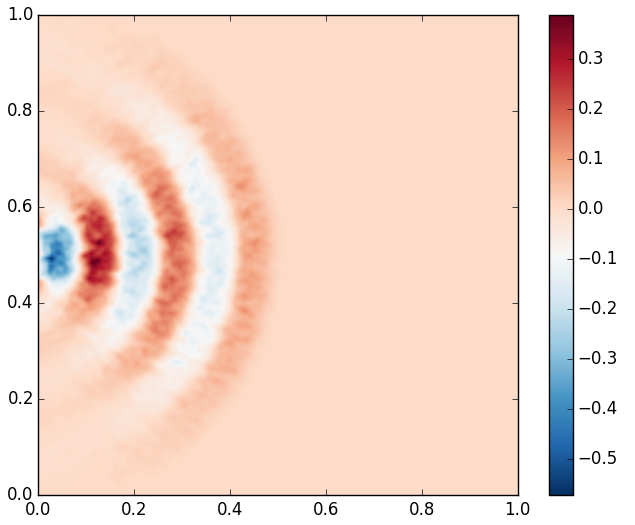}
}
&
    \subfigure[Velocity in y-direction.]{
\includegraphics[width=0.31\textwidth]{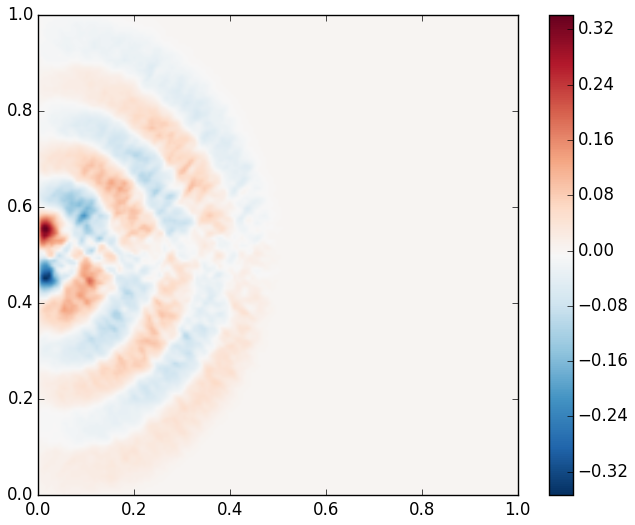}
}
\end{tabular}
\caption{
    A sample solution of the linear acoustics equation with a stochastic background velocity field.
    The solution shows that the background velocity field distorts the propagating waves.
}
\protect \label{fig:linearAccSamples}
\end{figure}

\noindent
Defining the sound speed as $c_0 = \sqrt{K_0/\rho_0}$, the eigensystem of \eqref{eq:linear_acoustics} is given by
\begin{align}
    \begin{split}
        \lambda_{1,3}^\w(\X,t,\omega) &= \breve{u}^\w_0(x,t,\omega) \mp c_0,\quad
        \lambda_2^\w(\X,t,\omega) = \breve{u}^\w_0(x,t,\omega), \\
        \Q^\w  &= 
    \begin{pmatrix}
        -\rho_0 c_0  & 0 & \rho_0 c_0 \\
        \text{w}_x & -\text{w}_y & \text{w}_x \\
        \text{w}_y & \text{w}_x & \text{w}_y \\
    \end{pmatrix}, \\
    \end{split}
\end{align}
where the rows of $\Q^\w$ consist of the right eigenvectors of $\A^\w$.
The eigenvectors are deterministic and are in fact the same as in the deterministic case. Therefore, the bound in Equation~\eqref{eq:bound} naturally holds.
Furthermore, we have the following estimates
\begin{align}
    \begin{split}
        |\E[\lambda^\w_{1,3} (\X,t,\omega)]| &\leq 2c_0 + |\E[u_0(\X,t,\omega)]| + |\E[v_0(\X,t,\omega)]|, \\
        |\E[\lambda^\w_2 (\X,t,\omega)]| &\leq |\E[u_0(\X,t,\omega)]| + |\E[v_0(\X,t,\omega)]|.
    \end{split}
\end{align}
For $u_0,v_0$ defined as in~\eqref{eq:timedependentRF} these expectations are given by
\begin{align}
    \label{eq:boundEtimeRF}
    \begin{split}
    \E[u_0(\X,t,\omega)] = \mu_u(\X) + (a_0(\X)-\mu_u(\X))e^{-\theta t}, \\
    \E[v_0(\X,t,\omega)] = \mu_v(\X) + (b_0(\X)-\mu_v(\X))e^{-\theta t}.
    \end{split}
\end{align}
Using this, it is easy to show the bound \eqref{eq:eigenvaluebound}, and therefore that the system~\eqref{eq:linear_acoustics} is hyperbolic according to Definition~\ref{def:hyperbolicity}.

\begin{figure}
    \centering
    \begin{tabular}{lr}
            \subfigure[Mean of pressure.]{
        \includegraphics[width=0.48\textwidth]{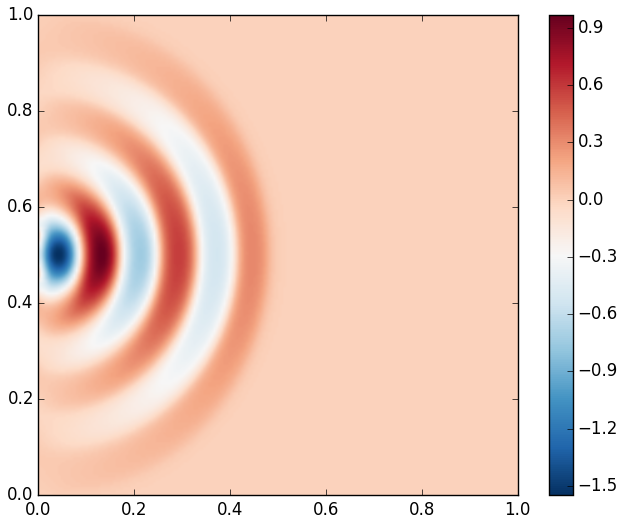}
        }
        &
        \subfigure[Variance of pressure.]{
        \includegraphics[width=0.48\textwidth]{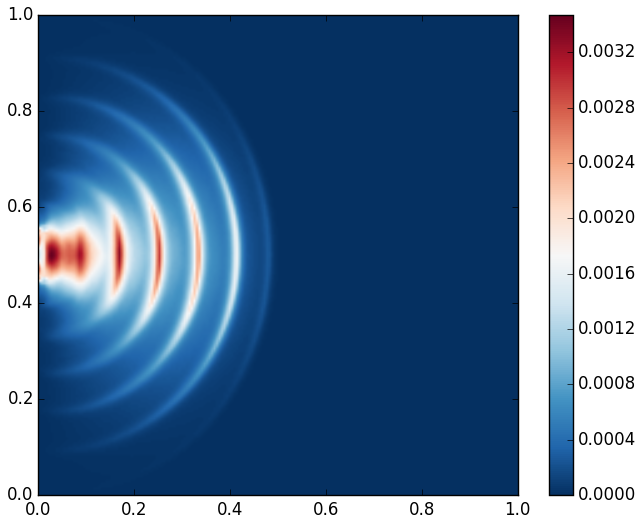}
        }
        \\
        \subfigure[Mean of velocity in y-direction.]{
        \includegraphics[width=0.48\textwidth]{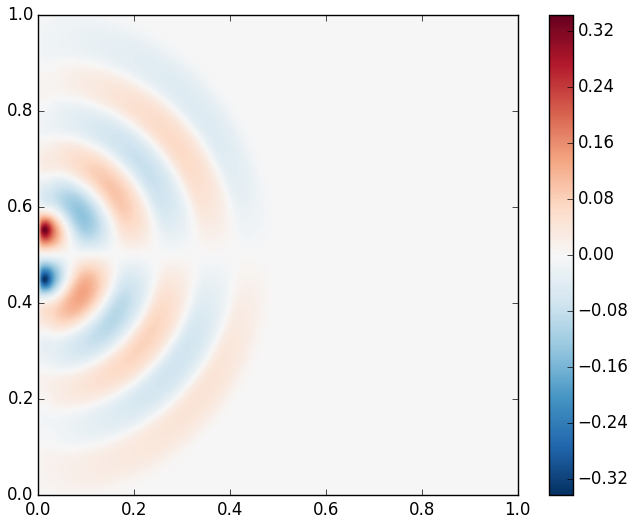}
        }
        &
        \subfigure[Variance of velocity in y-direction.]{
        \includegraphics[width=0.48\textwidth]{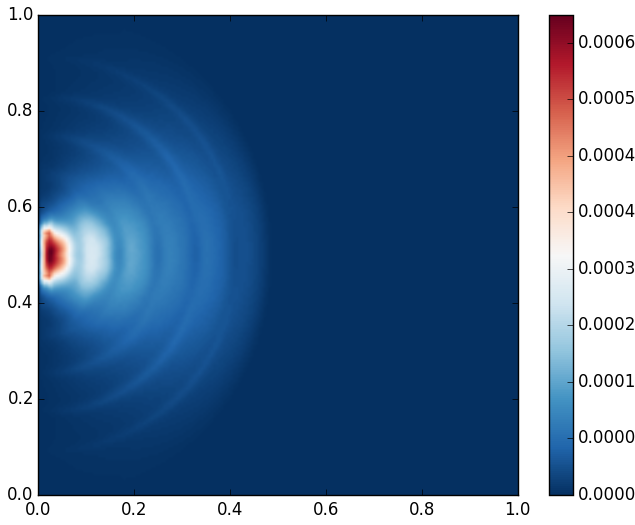}
        }
    \end{tabular}
\caption{
    Statistic mean and variance of a propagating sound wave computed with the proposed approach at time $t=1.5$.
    Waves are introduced on the left boundary and propagate radially through the domain, showing a symmetric structure, consisting of smooth circular wave fronts for pressure.
    The variances are observed where the mean of the solution changes sign.
    }
\protect \label{fig:UQlinearAcoustics}
\end{figure}

\begin{table}
\setlength{\tabcolsep}{.8em}
    \begin{tabular}{cc}
\setlength{\tabcolsep}{.0em}
    \begin{tabular}{c|cccc}
        \multicolumn{5}{c}{First order MC FV scheme}\\
        \hline
        $\Delta x = \Delta y$ & $1/16$ & $1/32$ & $1/64$ & $1/128$\\
        \hline
        $E_M[p]$ &   99.7  \% & 97.2 \% & 87.7 \% & 64.3 \% \\
        $E_M[u]$ &   99.3  \% & 96.6 \% & 87.3 \% & 64.4 \% \\
        $E_M[v]$ &   101.9 \% & 99.3 \% & 87.4 \% & 62.8 \% \\

        \hline
        \hline
        \multicolumn{5}{c}{Second order MC FV scheme}\\
        \hline
        $E_M[p]$ &   98.8  \% & 90.9 \% & 60.8 \% & 26.7 \% \\
        $E_M[u]$ &   98.6  \% & 90.0 \% & 61.3 \% & 27.6 \% \\
        $E_M[v]$ &   101.1 \% & 91.6 \% & 58.7 \% & 24.9 \% \\

    \end{tabular}
    &
\setlength{\tabcolsep}{.0em}
    \begin{tabular}{c|cccc}
        \multicolumn{5}{c}{First order MC FV scheme}\\
        \hline
        $\Delta x = \Delta y$ & $1/16$ & $1/32$ & $1/64$ & $1/128$\\
        \hline
        $V_M[p]$ & 99.8  \% & 98.8 \% & 96.0 \% & 85.2 \% \\
        $V_M[u]$ & 99.9  \% & 99.4 \% & 97.8 \% & 91.0 \% \\ 
        $V_M[v]$ & 100.0 \% & 99.8 \% & 98.6 \% & 91.9 \% \\
        \hline
        \hline
        \multicolumn{5}{c}{Second order MC FV scheme}\\
        \hline
        $V_M[p]$ &   92.8 \% & 93.8 \% & 72.5 \% & 40.1 \% \\
        $V_M[u]$ &   97.2 \% & 93.2 \% & 79.7 \% & 61.0 \% \\
        $V_M[v]$ &   99.6 \% & 97.5 \% & 86.6 \% & 63.1 \% \\
    \end{tabular}
    \end{tabular}
    \caption{(Self) convergence to the solution with the second order scheme with $\Delta x = \Delta y = 1/256$.}
    \protect \label{table:convLinacc}
\end{table}

We test our approach for an initial condition given by
\begin{equation}
    \U_0(\X) = (p(\X),u(\X),v(\X))^T = (0,0,0)^T.
\end{equation}
We point out that pressure values do not have to be positive, because the system~\eqref{eq:linear_acoustics} is derived by linearizing around a background state $(p_0,u_0,v_0)$.
This means that $p,u$, and $v$ describe the perturbation relative to the background state.
The domain is given by $D = [0,1]^2$ with Neumann boundary conditions on the right, top and bottom boundary. On the left boundary we have an acoustic source given by
\begin{equation}
    \U(0,y,t) = 
    \begin{cases}
        (0, \sin(4\pi t), 0)^T, & \text{ if } |y-\frac{1}{2}|<0.05 \\
        (0, 0, 0)^T, & \text{ otherwise.}
    \end{cases}
\end{equation}
The covariance for the stochastic background velocities $u_0,v_0$ is given by the Lebesgue density \eqref{eq:correlation} with $q=2$ and $l=4$.
For the initial condition of Equation~\eqref{eq:timedependentRF} we use $Z(\X,0) = 0$, for the mean $\mu(\X) = 0$, and the remaining parameters are set to $\theta = \sigma = 1$.

As a numerical scheme we use an approximate Riemann solver of the HLL-type (see~\cite{LEV1}),
resolving the outermost waves.
Let $\U_{L,R}$, and $\mathbf{F}_{L,R}$ denote the left and right state and flux respectively.
Then the numerical flux is given by
\begin{equation}
    \mathbf{F}^\text{HLL}(\U_L, \U_R) =
    \begin{cases}
        \mathbf{F}_L, & \text{ if } \frac{x}{t} \leq s_L,\\
        \mathbf{F}_*, & \text{ if } s_L < \frac{x}{t} < s_R,\\
        \mathbf{F}_R, & \text{ if } s_R \leq \frac{x}{t},\\
    \end{cases}
\end{equation}
where $s_L = \lambda_1$, $s_R = \lambda_3$, and $\mathbf{F}_*$ is determined from conservation leading to
\begin{equation}
    \mathbf{F}_* = \frac{s_R \mathbf{F}_L - s_L \mathbf{F}_R + s_L s_R (\U_R - \U_L)}{s_R - s_L}.
\end{equation}

For first and second order MC FV methods, approximations of the random background velocity field~\eqref{eq:linearAccVelField} are based on the Milstein scheme, and a (weak) second order stochastic Runge-Kutta scheme, respectively.
For the given parameters, the discrete time interval of the OU--process is $h = \frac{1}{2 c_0} \min\{\Dx,\Dy\}$.

Results of the \emph{deterministic} FV simulation are given in Figure~\ref{fig:linearAccSamples}.
For the given scenario, a sample solution is plotted at time $t=1.5$ with a second order minmod scheme on a $256\times256$ mesh.
We can see that the waves enter the domain at the left boundary and propagate radially through the domain.
As expected, the waves are distorted due to the presence of the random background velocity field.

For the \emph{stochastic} MC FV simulation we compute the mean and variance of the solution with a scheme of order $o$ using $M=\Ord\left((1/\Delta x)^{-2o}\right)$ Monte Carlo samples (see Equation~\eqref{eq:mc_samples}).
The structure of the mean of the propagating waves shown in Figure~\ref{fig:UQlinearAcoustics} resembles the structure of the waves seen in the deterministic simulation of one sample shown in Figure~\ref{fig:linearAccSamples}.
The sound waves introduced on the left boundary travel radially through the domain, showing a symmetric structure, consisting of smooth circular wave fronts for pressure.
The largest values of the variance of a conserved quantity are observed around sign changes of the mean of the solution of that quantity.
Table~\ref{table:convLinacc} shows self-convergence of the proposed MC FV method to a reference solution on a $256\times256$ grid.

\subsection{Magnetic Induction Equation in 2 dimensions}
\begin{figure}
\centering
\begin{tabular}{lr}
    \subfigure[Magnetic field in x-direction.]{
\includegraphics[width=0.48\textwidth]{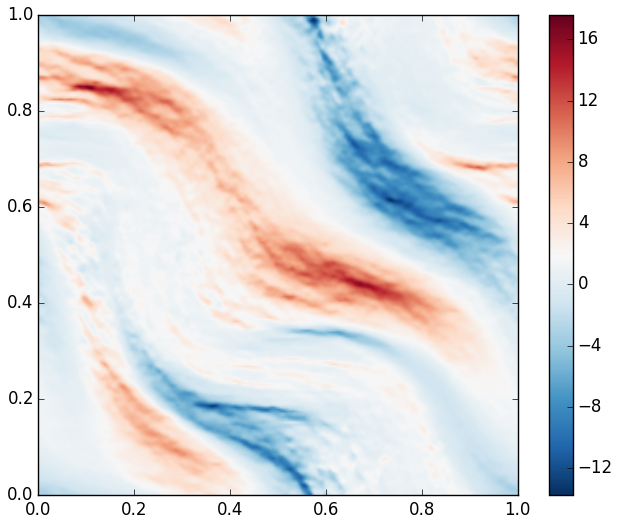}
}
&
\subfigure[Magnetic field in y-direction.]{
\includegraphics[width=0.48\textwidth]{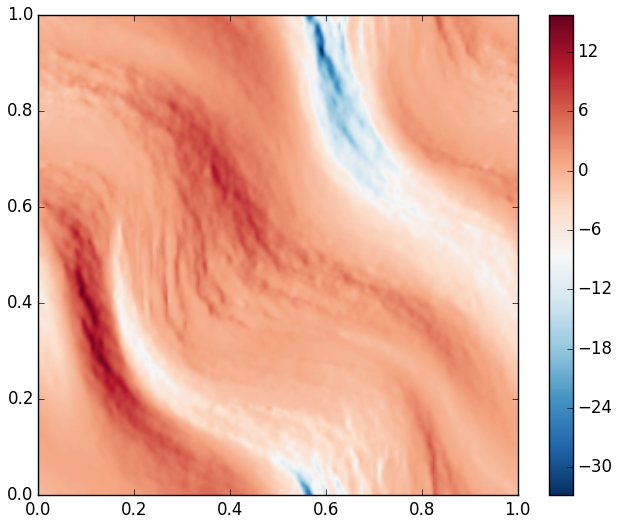}
}
\\
\subfigure[Velocity field in x-direction.]{
\includegraphics[width=0.48\textwidth]{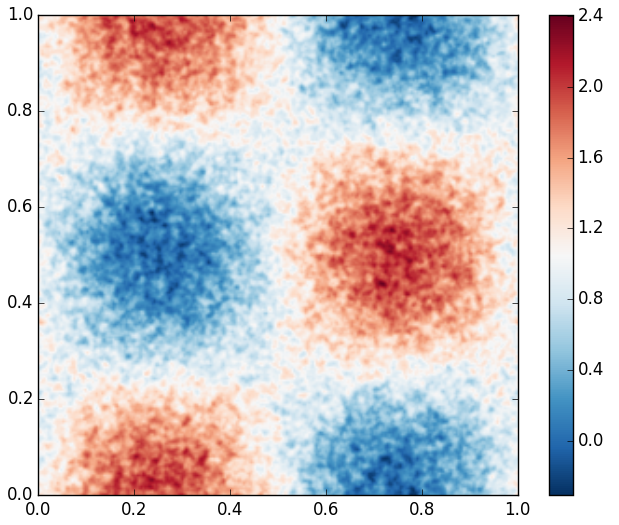}
}
&
\subfigure[Velocity field in y-direction.]{
\includegraphics[width=0.48\textwidth]{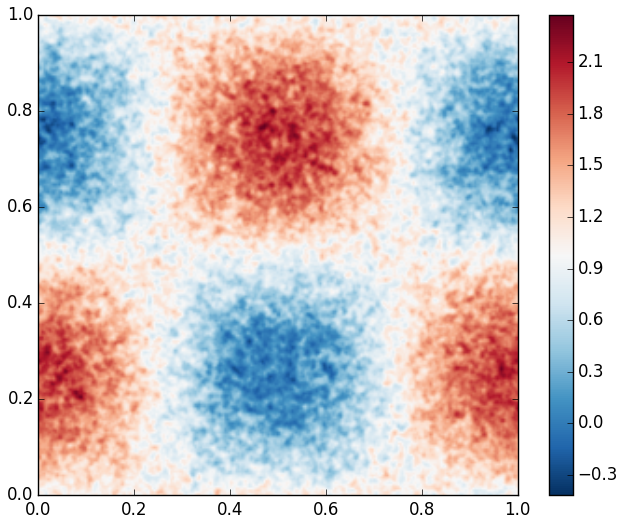}
}
\end{tabular}
\caption{
    Sample solution at time $t=0.75$ for the magnetic induction equation.
    The initial magnetic field has been advected and distorted as a result of the stochastic velocity field.
}
\protect \label{fig:InductionEQsample}
\end{figure}

The magnetic induction equation describes the evolution of a magnetic field $\U$ for a given velocity field $\V$.
We use the symmetric form of the equations, see \cite{fuchs2009stable}, given by
\begin{equation}\label{eq:induction_equation}
    \begin{cases}
        \partial_t \U + \Div (\V \otimes \U - \U \otimes \V) =  -\V \div(\U), \\
        \hfill \U(\x,0) = \U_0(\x), \text{ with } \Div(\U_0) = 0.
    \end{cases}
\end{equation}
The equations have the intrinsic constraint that the divergence of the magnetic field $\U$ is preserved in time, i.e., $\partial_t(\Div \U) = 0$.
Therefore, the system \eqref{eq:induction_equation} is analytically equivalent to the conservative form without any source term.
We consider the equation in two dimensions, where the components of the velocity field $\V=(u,v)$ are given by spatiotemporal random fields as defined in Equation~\eqref{eq:timedependentRF}, i.e., as the solution of the following SDEs
\begin{align}\label{eq:magindVel}
    \begin{split}
         d u(\x,t) &= \theta_u(\mu_u(\x) - u(\x,t)) dt + \sigma_u dG_u(\x,t),\\
         d v(\x,t) &= \theta_v(\mu_v(\x) -\, v(\x,t)) dt + \sigma_v dG_v(\x,t).
    \end{split}
\end{align}

The eigensystem of the symmetric system~\eqref{eq:induction_equation} is given by
\begin{equation}
    \lambda^\w_{1,2}(\X,t,\omega) = \V(\X,t,\omega) \cdot \w, \quad
        \Q^\w  = 
    \begin{pmatrix}
        1 & 0\\
        0 & 1
    \end{pmatrix}. \\
\end{equation}
It is easy to show that the system is hyperbolic according to Definition~\ref{def:hyperbolicity}.
The bound~\eqref{eq:bound} trivially holds, as $Q$ is the matrix identity.
Furthermore, we have that 
\begin{equation}
        |\E[\lambda^\w_{1,2} (\X,t,\omega)]| \leq |\E[u_0(\X,t,\omega)]| + |\E[v_0(\X,t,\omega)]|.
\end{equation}
Using \eqref{eq:boundEtimeRF}, we can easily show the bound \eqref{eq:eigenvaluebound} on the eigenvalues.

Numerical schemes approximating the solutions of the induction equation~\eqref{eq:induction_equation} have to address the divergence constraint.
Here, we will use the "stable upwind scheme" presented in \cite{fuchs2009stable}.
To test our approach, we consider the equation on the domain $D = [-\frac{1}{2},\frac{1}{2}]^2$, with periodic boundary conditions, and a divergence free initial magnetic field, given by a potential function $A$, i.e.,
\begin{equation}
    \U(x,y) = (\partial_y A(x,y), \partial_x A(x,y)),\;\text{ with } A(x,y) = \frac{1}{2\pi} \sin(2\pi x) \sin(2\pi y) + y -x.
\end{equation}
The mean of the background velocity $\V(\X,t,\omega)$, see Equation~\eqref{eq:timedependentRF}, is defined as
\begin{equation}
    \mu(x,y) = \left(1+\frac{\cos(2\pi x) + 2 \sin(2\pi y)}{4},
                     1+\frac{\sin(2\pi x) + 2 \cos(2\pi y)}{4}\right).
\end{equation}
We set the initial condition of the velocity field, i.e., of Equation~\eqref{eq:timedependentRF}, to equal the mean, i.e., $Z(x,y) = \mu(x,y)$, 
and the remaining parameters to $\theta = 1$, and $\sigma = 10$.
The covariance is given by the Lebesgue density \eqref{eq:correlation} with $q=2$ and $l=4$.
The approximations of the random background velocity field~\eqref{eq:magindVel} are based on the Milstein scheme.
For the given parameters, the discrete time interval of the OU--process is $h = \frac{1}{4} \min\{\Dx,\Dy\}$.

Results of the \emph{deterministic} FV simulation are shown in Figure~\ref{fig:InductionEQsample}, together with the sample vector field.
The first order stable upwind scheme is used to approximate the solution at time $t=0.75$ on a $256\times256$ mesh.
We see that the initial magnetic field has been advected and distorted due to the space- and time-dependent velocity field.

For the \emph{stochastic} MC FV simulation we compute the mean and variance of the solution using $M=100\left(\frac{1}{\Delta x}\right)^{-1}$ Monte Carlo samples (see Equation~\eqref{eq:mc_samples}).
The structure of the mean of the propagating waves shown in Figure~\ref{fig:UQinductionEQ} resembles the structure of the waves seen in the deterministic simulation of one sample shown in Figure~\ref{fig:InductionEQsample}.
The values of the variance exhibit an interesting structure, which will be analyzed in a forthcoming paper.
The table in Figure~\ref{fig:UQinductionEQ} shows the expected convergence rate for the first order scheme, for both the first and the second statistical moment.
The second statistical moment has a larger error constant compared to the first moment.

\begin{figure}
\centering
\begin{tabular}{lr}
    \subfigure[Mean of magnetic field in x-direction.]{
\includegraphics[width=0.48\textwidth]{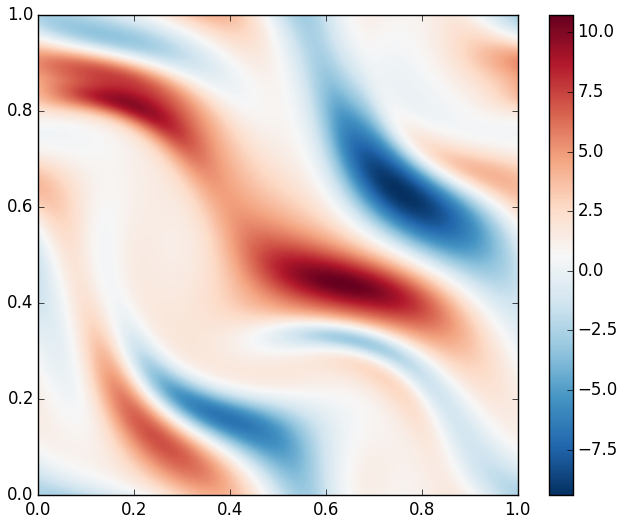}
}
&
\subfigure[Variance of magnetic field in x-direction.]{
\includegraphics[width=0.48\textwidth]{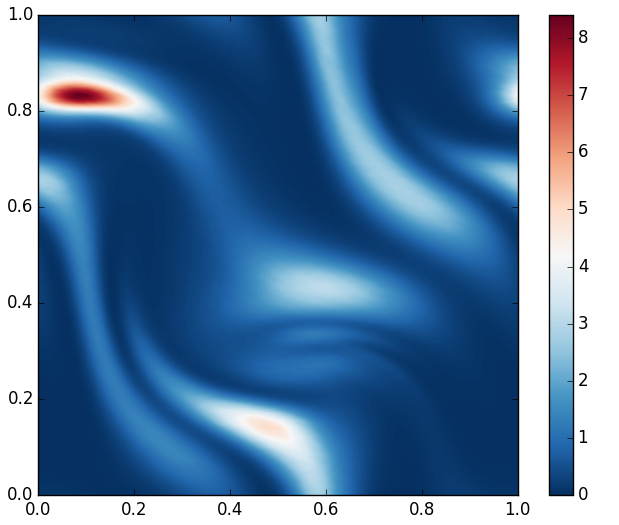}
}
\\
\subfigure[Mean of magnetic field in y-direction.]{
\includegraphics[width=0.48\textwidth]{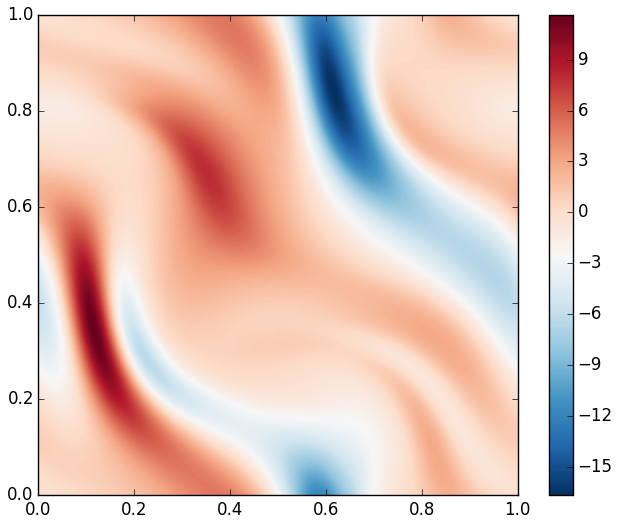}
}
&
\subfigure[Variance of magnetic field in y-direction.]{
\includegraphics[width=0.48\textwidth]{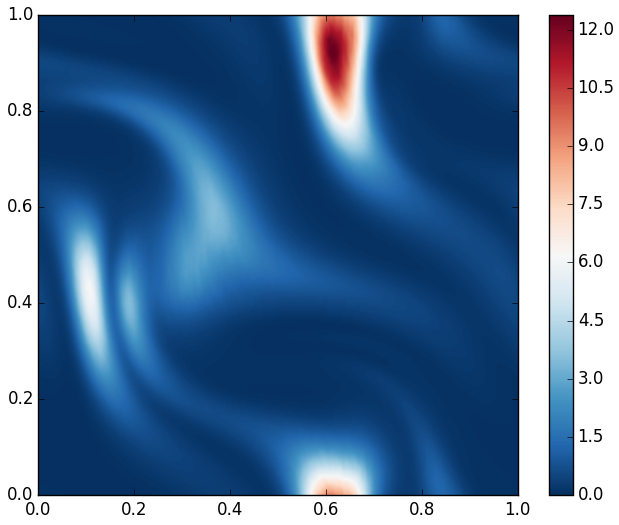}
}
\\
&\\
\multicolumn{2}{c}{Self convergence to solution with $\Delta x = \Delta y = 1/256$}\\
\multicolumn{2}{c}{
\setlength{\tabcolsep}{.5em}
\begin{tabular}{c|cccc}
    $\Delta x = \Delta y$ & $1/16$ & $1/32$ & $1/64$ & $1/128$\\
    \hline
    $E_M[U_0]$ & 73.0 \% & 60.0 \% & 41.9 \% & 19.6 \% \\
    $E_M[U_1]$ & 79.9 \% & 61.3 \% & 41.1 \% & 18.8 \% \\
    \hline
    \hline
    $\Delta x = \Delta y$ & $1/16$ & $1/32$ & $1/64$ & $1/128$\\
    \hline
    $V_M[U_0]$ & 98.8 \% & 97.7 \% & 89.9 \% & 62.5 \% \\
    $V_M[U_1]$ & 98.9 \% & 97.1 \% & 88.8 \% & 82.0 \% \\
\end{tabular}
}
\end{tabular}
\caption{
    Results of the computation of mean and variance using the first order stable upwind scheme.
    The magnetic field shows interesting features.
    The first moments of the solution converge to the reference solution at the expected rate.
}
\protect \label{fig:UQinductionEQ}
\end{figure}

\section{Conclusions and Outlook}\label{sec:conclusion}
Linear systems of hyperbolic conservation laws with random coefficients are considered. 
Those coefficients are modeled as time-dependent random fields, leading to a coupled system consisting of the conservation law and stochastic differential equations for each of the parameters.
An appropriate solution concept is developed and a Monte Carlo based algorithm is presented to approximate statistical moments of the solution.
Important examples are presented, namely linear acoustics and magnetic induction with random velocity field coefficients.
The results reveal interesting structures in the moments of the solution.
Error and convergence analysis validate the proposed method.

In the future, we plan to establish a rigorous error analysis and formal convergence proofs.
Furthermore, we will increase efficiency of the approach by developing highly parallel Multi Grid MC FV methods, utilizing the power of CPUs on coarse grid levels and graphics processing units (GPUs) on fine grid levels.
This will further facilitate simulations of more complex problems in three dimension, and problems where the time-dependent random fields are given by more complicated SDEs, or even given by stochastic partial differential equations (SPDEs).

\section*{Acknowledgement}
The authors would like to express their gratitude towards the Center of Mathematics for Applications (CMA) at the University of Oslo, the Seminar for Applied Mathematics at the Eidgen\"ossische Technische Hochschule Z\"urich (ETH) and SINTEF Oslo. The research of A. Barth leading to these results has further received funding from the German Research Foundation (DFG) as part of the Cluster of Excellence in Simulation Technology (EXC 310/2) at the University of Stuttgart, and it is gratefully acknowledged.

\bibliographystyle{siam}
\bibliography{references}
\end{document}